\documentclass[reqno]{amsart}
\usepackage{amsmath,amssymb}
\usepackage{amsfonts,amsthm}
\usepackage[utf8]{inputenc}

\usepackage{enumitem}
\usepackage{todonotes,xcolor}
\usepackage{graphicx}

\usepackage[foot]{amsaddr}
\definecolor{dkblue}{RGB}{1,31,91}
\theoremstyle{definition}
\newtheorem{theorem}{Theorem}

\newtheorem{lemma}[theorem]{Lemma}

\newtheorem{remark}[theorem]{Remark}

\usepackage[colorlinks=true, linkcolor=dkblue, citecolor=dkblue, urlcolor=dkblue]{hyperref} 

\newcommand{\C}{\mathbb{C}}

\newcommand{\R}{\mathbb{R}}

\newcommand{\norm}[2][]{\left\|#2\right\|_{#1}}

\newcommand{\Lpnorm}[2][]{\ifthenelse{\equal{#1}{}}{\norm{#2}_{L^p}}{\norm{#2}_{L^p(#1)}}}
\newcommand{\Hknorm}[2][]{\ifthenelse{\equal{#1}{}}{\norm{#2}_{H^k}}{\norm{#2}_{H^k(#1)}}}

\newcommand{\QV}[2][]{\ifthenelse{\equal{#1}{}}{\langle #1 \rangle}{\langle #1,#2 \rangle}}

\newcommand{\ind}{\mathbb{I}}

\begin{document}


\title[Vanishing angular singularity limit]{Vanishing angular singularity limit to the hard-sphere Boltzmann equation}

\author[J. W. Jang]{Jin Woo Jang}
\address{Department of Mathematics, Pohang University of Science and Technology (POSTECH), Pohang, South Korea (37673). \href{mailto:jangjw@postech.ac.kr}{jangjw@postech.ac.kr} }

\author[B. Kepka]{Bernhard Kepka}
\address{Institute for Applied Mathematics, University of Bonn, 53115 Bonn, Germany. \href{mailto:kepka@iam.uni-bonn.de}{kepka@iam.uni-bonn.de} }

\author[A. Nota]{Alessia Nota}
\address{Department of Information Engineering, Computer Science and Mathematics at Università degli Studi dell'Aquila (UnivAq), 67100  L'Aquila, Italy. \href{mailto:alessia.nota@univaq.it}{alessia.nota@univaq.it} }

\author[J. J. L. Vel\'{a}zquez]{Juan J. L. Vel\'{a}zquez}
\address{Institute for Applied Mathematics, University of Bonn, 53115 Bonn, Germany. \href{mailto:velazquez@iam.uni-bonn.de}{velazquez@iam.uni-bonn.de} }

\begin{abstract}
	In this note we study Boltzmann's collision kernel for inverse power law interactions $U_s(r)=1/r^{s-1}$ for $s>2$ in dimension $ d=3 $. We prove the limit of the non-cutoff kernel to the hard-sphere kernel and give precise asymptotic formulas of the singular layer near $\theta\simeq 0$ in the limit $ s\to \infty $. Consequently, we show that solutions to the homogeneous Boltzmann equation converge to the respective solutions.
\end{abstract}

\thispagestyle{empty}

\maketitle
\tableofcontents
\newpage

\section{Introduction}
The Boltzmann equation reads as
\begin{equation}\label{boltzmann eq}
    \partial_t f +v\cdot \nabla_x f = Q(f,f)(v),
\end{equation}where $f=f(t,x,v)$ is the velocity distribution of particles with position $x\in \Omega \subset \mathbb{R}^3$ and velocity $v\in\mathbb{R}^3$ at time $t\in[0,\infty)$.

The equation has been considered as a fundamental model for the collisional gases that interact either under the hard-sphere potential  $U_s(r)=\infty $ for $r\le 2\epsilon$ and $=0$ for $r\ge 2\epsilon$, or  under the long-range potential $U_s(r)\simeq\frac{1}{r^{s-1}}$ for $s>2$. Here $\epsilon$ is the radius of each hard-sphere.   The prototype of the model was suggested by Maxwell \cite{Maxwell1,Maxwell2} and Boltzmann \cite{Boltzmann1970}. 

In this note we consider the particular case of inverse power law interactions $ U_s(r)= 1/r^{s-1} $ leading to non-cutoff kernels (cf. formula \eqref{eq:CollisionKernel})
\begin{align*}
	B_s(|v-v_*|,\cos\theta) = |v-v_*|^{\gamma}b_s(\cos\theta),\quad \gamma=\dfrac{s-5}{s-1}.
\end{align*}
Here, $ b_s $ is the so-called angular part. We prove that the function $ B_s $ converges to the hard-sphere kernel in the limit  $ s\to \infty $. We give a precise study of the singularity as $ \theta\to 0 $ when $ s\to \infty $. 
Finally, we show that solutions to the homogeneous Boltzmann equation with collision kernel $ B_s $ converge to the solution to the equation for hard-spheres. Such a limit result was suggested to exist in \cite[Remark 1.0.1]{GallagherSaintRaymond2013NewtoToBoltzmann}.

\subsection{Boltzmann collision operator}
The Boltzmann collision operator $Q$ takes the form
\begin{align*}
		Q(f,f)(v)= \int_{\R^3}\int_{S^2} B(|v-v_*|,n\cdot \sigma) (f'f'_*-ff_*)\, d\sigma dv_*, \quad n:= \dfrac{v-v_*}{|v-v_*|},
\end{align*}
where we used the standard notation $ f'=f(v'), \, f'_*= f(v'_*),\, f_*=f(v_*) $. Also $ (v',v'_*) $ are the post-collisional velocities and $ (v,v_*) $ the pre-collisional velocities. The function $ B $ is Boltzmann's collision kernel and strongly depends on the microscopic interaction of two particles in the course of a collision. It only depends on the length of relative velocities $ |v-v_*| $ and the so-called deviation angle $ \theta \in [0,\pi] $ through $ n\cdot \sigma = \cos\theta $.

It is customary to distinguish two main classes of kernels, namely angular cutoff and non-cutoff kernels. This refers to a possible singularity of the kernel when $ \theta\to 0 $. Such deviation angles correspond to grazing collisions, i.e. collisions such that $ v\approx v' $. They appear only for long-range or weak interactions.

\subsection{Derivation of Boltzmann's collision kernel for long-range interactions}
Let us give here a derivation of the collision kernel for inverse power law interactions. We consider the collision of two particles $ (x,v) $, $ (x_*,v_*) $ with equal mass $ m=1 $. Due to conservation of momentum and conservation of energy, both $ v_c=(v+v_*)/2 $ and $ |v-v_*| $ are conserved. Here, $ v_c $ is the velocity of the center of mass $ x_c=(x+x_*)/2 $. It is convenient to use the coordinate system $ (\bar{x}, \bar{v})=(x-x_*,v-v_*) $, in which the center of mass is zero and at rest. In this coordinate system, the velocities after the collisions have equal lengths but opposite directions due to the conservation of momentum and energy. Hence, they are given by $ |\bar{v}|\sigma/2 $ and  $ -|\bar{v}|\sigma/2 $, respectively, for $ \sigma\in S^2 $. In the original coordinate system, we thus get
\begin{align*}
	v'= \dfrac{v+v_*}{2}+\dfrac{|v-v_*|}{2}\sigma, \quad v'_*= \dfrac{v+v_*}{2}-\dfrac{|v-v_*|}{2}\sigma.
\end{align*}

In order to derive the distribution of $ \sigma $ in the scattering problem, we need to consider the interaction of both particles via the potential $ U $. As is well-known we can reduce it to a single particle problem in the center of mass coordinate system $ (\bar{x}, \bar{v}) $ with (reduced) mass $ \mu=1/2 $, see e.g. \cite[Section 13]{Landau}. The motion is planar and we can use polar coordinates. The Hamiltonian reads,
\begin{align*}
	H(r,\varphi,\dot{r}, \dot{\varphi}) = \dfrac{\mu}{2}\left( \dot{r}^2+r^2\dot{\varphi}^2\right) + U(r).
\end{align*}
Both energy $ E= H(r,\varphi,\dot{r}, \dot{\varphi}) $ and angular momentum $ L=\mu r^2\dot{\varphi}  $ are conserved. 

For the collision process we consider the particle $ (\bar{x},\bar{v})(t) $ passing the center of the potential with asymptotic velocity $ v-v_* $ as $ t\to -\infty $, $ r\to \infty $. The particle is scattered and moves away from the center with asymptotic velocity $ v'-v'_*$ as $ t\to \infty $, $ r\to \infty $. The turning point ($ \dot{r}=0 $) is given at distance $ r_m $, which is the largest root of
\begin{align*}
		E - \dfrac{L^2}{r_m^2} - U(r_m) =0.
\end{align*}
We can determine $ E $ and $ L $ by considering the asymptotic value $ t\to -\infty $. This yields $$ E= \frac{|v-v_*|^2}{4} \text{ and } L=\mu|\bar{x}\times \bar{v}| =\frac{ |\bar{x}||\bar{v}|\sin (\psi) }{2} = \frac{|v-v_*| \rho }{2} ,$$ where $ \psi $ is the angle between $ \bar{x} $ and $ \bar{v} $. Furthermore, $ \rho $ is the impact parameter, which is the distance of the closest approach if the particle is passing the center without the presence of an interaction, see Figure \ref{fig:scattering}. The formula for $ L $ can be obtained by a geometric argument.

\begin{figure}[h]
	\centering
	\includegraphics[width=0.8\linewidth]{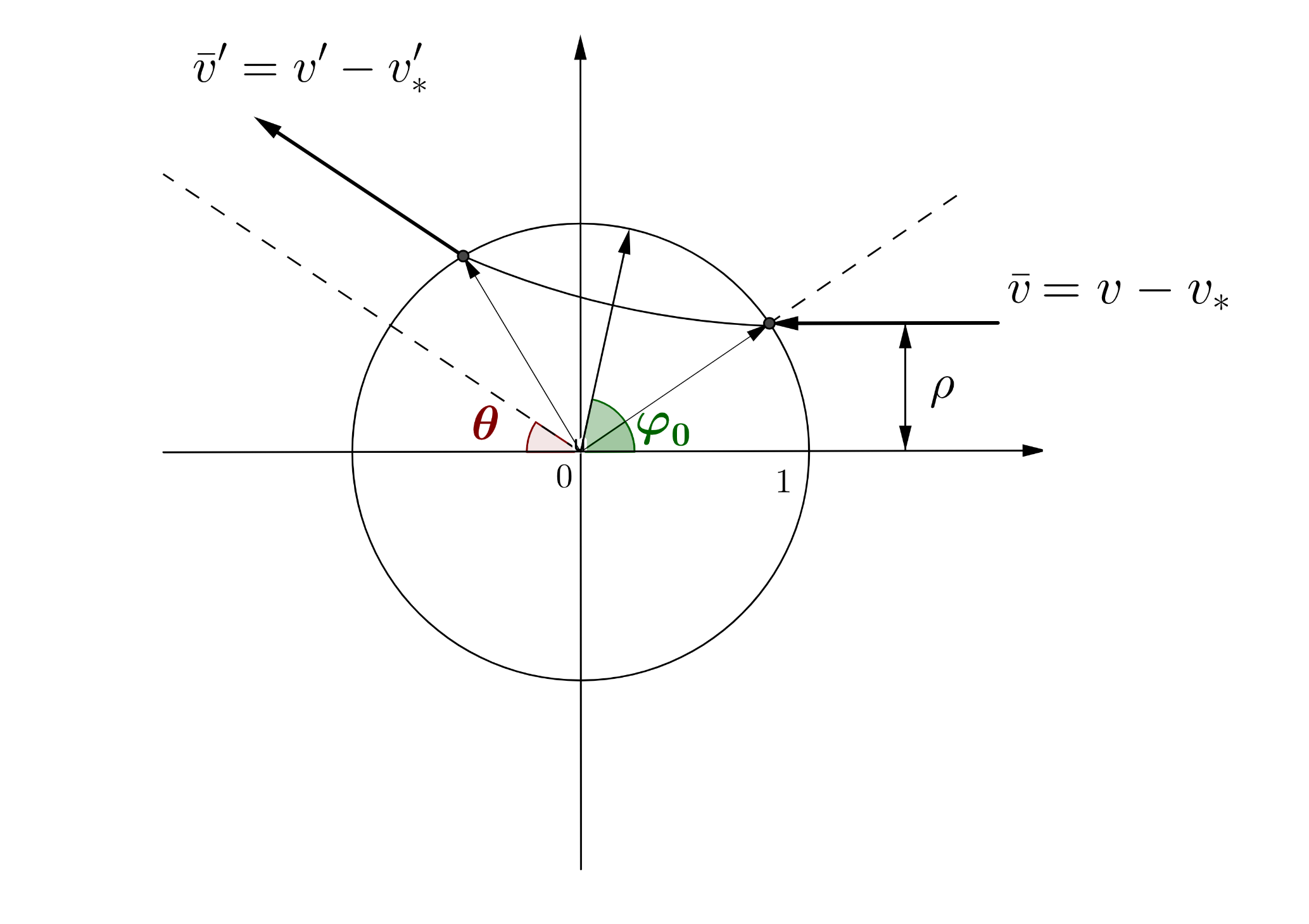}
	\caption{Two-body scattering process: $ \rho $ is the impact parameter, $ \theta $ the deviation angle and $ \varphi_0 $ the angle of the axis of symmetry.}
	\label{fig:scattering}
\end{figure}

The solution to the above problem is implicitly given by, see e.g. \cite[Section 14]{Landau},
\begin{align*}
		\varphi = \text{const.}+ \int_{r_m}^r \dfrac{L/r_*^2\, dr_*}{\sqrt{E-U(r_*)-\frac{L^2}{r_*^2}}}, \quad t = \text{const.}+\int_{r_m}^r \dfrac{dr_*}{2\sqrt{E-U(r_*)-\frac{L^2}{r_*^2}}}.
\end{align*}
In the limit $ t\to-\infty $ the angle $ \varphi $ is zero. By a symmetry argument, one can see that the angle $ \varphi_0 $ of the line through the center and the point of closest approach is given by (see Figure \ref{fig:scattering})
\begin{align*}
	\varphi_0 = \int_{r_m}^\infty \dfrac{L/r_*^2\, dr_*}{\sqrt{E-U(r_*)-\frac{L^2}{r_*^2}}}.
\end{align*}
Now, we plug in the values for $ E, \, L $ and use the change of variables $ y=\rho/r_* $. Furthermore, we use $ U(r)=r^{-(s-1)} $ and define $ \beta = \rho (|v-v_*|/2)^{2/(s-1)} $ to get, cf. \cite[page 69-71]{Cercignani},
\begin{align}\label{eq:DefIntegral}
		\varphi_0 = \int_{0}^{x_0}\dfrac{dy}{\sqrt{1-y^2-(y/\beta)^{s-1}}}, \quad x_0 = \rho/r_m.
\end{align}
The deviation angle is given by $ \theta = \pi-2\varphi_0 $ for a given impact parameter $ \rho $.

The number of particles scattered with deviation angle close to $ \theta $ is proportional to $ |v-v_*| $ and the corresponding cross-section, that is $ 2\pi\rho d\rho= 2\pi \rho(\theta)|\rho'(\theta)|\, d\theta  $. Changing to the variable $ \beta $ and integrating via the solid angle yields the formula
\begin{align}\label{eq:CollisionKernel}
	B_s(|v-v_*|,\cos\theta)\,  d\sigma =  2^{\frac{4}{s-1}}|v-v_*|^{\frac{s-5}{s-1}}\dfrac{\beta(\theta)}{\sin \theta}|\beta'(\theta)| \, d\sigma.
\end{align}
This completes the formal derivation of the Boltzmann collision operator for the long-range interactions.

\subsection{Outline of the article}
We now provide a brief outline of the rest of the article. In Section \ref{sec.kernel}, we give a proof of the limit of the non-cutoff kernel to the hard-sphere kernel as $s\to \infty$. Then in Section \ref{sec.asymp}, we study the asymptotics of the singular layer near $\theta \simeq 0$ as $s\to \infty$. Finally, in Section \ref{sec:HomogBE}, we prove the convergence of the solution to the spatially homogeneous Boltzmann equation without angular cutoff to the solution to the hard-sphere Boltzmann equation as $s\to \infty$.

\section{Limit of the non-cutoff collision kernel}\label{sec.kernel} 
In this section, we study the limit of the kernel \eqref{eq:CollisionKernel} as $s\to \infty$.
Our first result contains the limit of the kernel as $s\to \infty$ as well as some uniform estimates. 
These estimates together with the ones in Section \ref{sec.asymp} play a crucial role for the proof of the rigorous limit of a weak solution to the spatially homogeneous Boltzmann equation without angular cutoff to the one for the hard-sphere interaction, see Section \ref{sec:HomogBE}. 

\begin{theorem}\label{thm:Kernel}
	Let us define the angular part of the collision kernel via $$ b_s(\cos\theta)=2^{4/(s-1)}\frac{\beta(\theta)}{\sin \theta}\beta'(\theta),\  s\geq2 .$$
	\begin{enumerate}[label=(\roman*)]
		\item We have as $ s\to\infty $
		\begin{align*}
				b_s(\cos\theta) \to \dfrac{1}{4}
		\end{align*}
	locally uniformly for $ \theta \in (0,\pi] $.
	\item The following asymptotics holds
	\begin{align*}
			\lim_{\theta\to 0} \theta^{1+2/(s-1)} \, b_s(\cos\theta) \sin\theta = C_s, 
			\quad
			C_s := \dfrac{2^{4/(s-1)}}{s-1} \left(\dfrac{\sqrt{\pi} \Gamma\left(\frac{s}{2}\right)}{\Gamma\left(\frac{s-1}{2}\right)}\right) ^{2/(s-1)}.
	\end{align*}
	\item Finally, we have the uniform bound
	\begin{align*}
			\sup_{s\geq 3}\, \sup_{\theta\in(0,\pi]} \theta^{1+2/(s-1)} \, b_s(\cos\theta) \sin\theta <\infty.
	\end{align*}
	\end{enumerate}
\end{theorem}
\begin{remark}\label{rem:Thm1}
	Note that in (i) the limiting collision kernel corresponds to hard-sphere interactions. Writing the kernel \eqref{eq:CollisionKernel} in terms of the angle $ \varphi = (\pi-\theta)/2 $ we get $ |v-v_*|\cos \varphi \, \ind_{\cos\varphi\geq0} $ as $ s\to \infty $.
	
	Furthermore, in (ii) we have $ C_s\to 0 $ as $ s\to \infty $. In fact,
	\begin{align*}
		\dfrac{\sqrt{\pi} \Gamma\left(\frac{s}{2}\right)}{\Gamma\left(\frac{s-1}{2}\right)} = \dfrac{s-1}{2}\dfrac{\sqrt{\pi}\Gamma\left(\frac{s}{2}\right)}{\Gamma\left(\frac{s+1}{2}\right)} = \dfrac{s-1}{2}B\left( \dfrac{s}{2},\dfrac{1}{2} \right) = (s-1)W_{s-1},
	\end{align*} 
	where $ W_{s-1} $ is the Wallis integral. It is known that $\lim_{s\to \infty} \sqrt{s}W_{s-1}=\sqrt{\pi/2} $.
	
	Finally, compare (iii) with \cite[Section 20]{Landau}.
\end{remark}
\subsection{Rearrangement of the deviation angle}
It is convenient to rearrange \eqref{eq:DefIntegral}
\begin{align}\label{eq:DefIntegral2}
	\varphi = x\int_0^1\dfrac{dz}{\sqrt{1-z^{s-1}-x^2(z^2-z^{s-1})}}.
\end{align}
Here, we dropped the index zero in $ \varphi_0, \, x_0 $, used the change of variables $ z=y/x_0 $ and the fact that $ x_0=x $ is the positive root of
\begin{align}\label{eq:RootEq}
	1-x^2-\dfrac{x^{s-1}}{\beta^{s-1}}=0.
\end{align}
We recall that the deviation angle $ \theta = \pi-2\varphi $. One can see that the mappings $ \beta\mapsto x, \, x\mapsto \varphi $ are strictly increasing and real analytic functions $ [0,\infty)\rightarrow[0,1)\rightarrow [0,\pi/2) $ for each $ s\geq2 $. We will use the index $ s $ to indicate that we consider the variable as a function.
\subsection{Proof of Theorem \ref{thm:Kernel}}
\begin{proof}[\textit{Proof of Theorem} \ref{thm:Kernel} \rm{(i)}] We first study the function $ \varphi_s(x) $. The integrand can be written
	\begin{multline*}
		\dfrac{1}{\sqrt{1-z^{s-1}-x^2(z^2-z^{s-1})}} = \dfrac{1}{\sqrt{1-z^{s-1}}\sqrt{1-x^2z^2+x^2z^2 \left( 1-\frac{1-z^{s-3}}{1-z^{s-1}}\right) }}
	\\	\leq \dfrac{1}{\sqrt{1-z}}\dfrac{1}{\sqrt{1-x^2z^2}}.
	\end{multline*}
	Here, we used that $$ 1-\frac{1-z^{s-3}}{1-z^{s-1}}\geq 0,\text { for } s\geq 3 .$$ This yields for any $ x\in \C $ with $ |x|\in [0,1-\varepsilon] $, $ \varepsilon>0 $ a uniform majorant, entailing locally uniform convergence, $$ s\to \infty,\  \varphi_s(x)\to \arcsin x.$$ As a consequence of the analyticity we have ($ x_s $ is the inverse of $ \varphi_s $)
	\begin{align*}
		x_s(\varphi) \to \sin\varphi \text{ and } x_s'(\varphi)\to \cos\varphi
	\end{align*}
	locally uniformly for $ \varphi \in [0,\pi/2) $.
	
	Next, we look at the functions (see \eqref{eq:RootEq})
	\begin{align*}
			\beta_s(x) = \dfrac{x}{(1-x^2)^{1/(s-1)}},
			\quad
			\beta_s'(x) = \dfrac{2}{s-1} \dfrac{1}{(1-x^2)^{s/(s-1)}} + \dfrac{s-3}{s-1}\dfrac{1}{(1-x^2)^{1/(s-1)}}.
	\end{align*}
	Hence, we have the locally uniform convergence for $ x\in[0,1) $ as $ s\to \infty $
	\begin{align*}
			\beta_s(x) \to x, \quad \beta_s'(x) \to 1.
	\end{align*}
	We conclude with the above analysis
	\begin{align}\label{eq:ProofLimitKernel}
			\begin{split}
				b_s(\cos\theta) &= \dfrac{1}{2}\, \dfrac{2^{4/(s-1)}}{\sin\theta} \, \beta_s\left( x_s\left( \dfrac{\pi-\theta}{2}\right) \right) \beta_s'\left( x_s\left( \dfrac{\pi-\theta}{2} \right) \right) x_s'\left( \dfrac{\pi-\theta}{2} \right) 
				\\
				&\to \dfrac{1}{2\sin\theta} \sin\left( \dfrac{\pi-\theta}{2} \right) \cos\left( \dfrac{\pi-\theta}{2}\right) = \dfrac{1}{4}
			\end{split}
	\end{align}
	locally uniformly for $ \theta\in (0,\pi] $ as $ s\to \infty $. Notice that $ \varphi=(\pi-\theta)/2 $ and the extra factor $ 1/2 $ results from $ d\varphi /d\theta =-1/2 $.
	\newline
	\newline
	\noindent\textit{Proof of Theorem} \ref{thm:Kernel} (ii). We have the following equalities for $ \varphi\in [0,\pi/2) $ and some $ \psi\in (\varphi,\pi/2) $
	\begin{align}\label{eq:ProofAsymp}
		\begin{split}
				1-x_s(\varphi) &= \varphi'_s(x_s(\psi))^{-1}\left( \dfrac{\pi}{2}-\varphi\right),
				\\
				\beta_s(x) &= \dfrac{x}{(1+x)^{1/(s-1)}} \, (1-x)^{-1/(s-1)},
				\\
				\beta_s'(x) &= \dfrac{2}{(s-1)(1+x)^{s/(s-1)}} \, (1-x)^{-s/(s-1)} + \dfrac{s-3}{s-1}\dfrac{(1-x)^{-1/(s-1)}}{(1+x)^{1/(s-1)}}.
		\end{split}
	\end{align}
	Combining them yields
	\begin{multline*}
			\lim_{\varphi\to \pi/2} \left(\dfrac{\pi}{2}-\varphi \right)^{(s+1)/(s-1)} \beta_s(x_s(\varphi)) \, \beta_s'(x_s(\varphi))\, x_s'(\varphi) 
			\\
			= \dfrac{1}{2^{1/(s-1)}}\varphi_s'(1)^{1/(s-1)}\, \dfrac{2}{s-1}\dfrac{1}{2^{s/(s-1)}}\, \varphi_s'(1)^{s/(s-1)} \varphi_s'(1)^{-1}\\ = \dfrac{1}{2^{2/(s-1)}}\dfrac{\varphi_s'(1)^{2/(s-1)}}{s-1}.
	\end{multline*}
	Let us note that
	\begin{align*}
		\varphi_s'(x) = \int_0^1\dfrac{1-z^{s-1}}{(1-z^{s-1}-x^2(z^2-z^{s-1}))^{3/2}}\, dz.
	\end{align*}
	and as a consequence we have
	\begin{align*}
			 \varphi_s'(1)=\int_0^1 \dfrac{1-z^{s-1}}{(1-z^2)^{3/2}}\, dz = \dfrac{\sqrt{\pi} \Gamma\left(\frac{s}{2}\right)}{\Gamma\left(\frac{s-1}{2}\right)}.
	\end{align*}
	Using a similar expression as in \eqref{eq:ProofLimitKernel} we get the asserted asymptotics.
	\newline
	\newline
	\noindent\textit{Proof of Theorem} \ref{thm:Kernel} (iii). For the last estimate we use \eqref{eq:ProofAsymp}. Note that $ \varphi_s' $ is increasing for $ s\geq3 $, so that
	\begin{align*}
		\sup_{\varphi \in [0,\pi/2)} x_s'(\varphi) = \varphi_s'(0)^{-1}.
	\end{align*}
	Note that
	\begin{align*}
		\varphi_s'(0) = \int_0^1 \dfrac{dz}{\sqrt{1-z^{s-1}}} \geq 1.
	\end{align*}
	The last inequality follows from the fact that $ s\mapsto \varphi_s'(0) $ is a decreasing function and $ \varphi_s'(0)\to 1 $ as $ s\to \infty $.
	This implies $ x_s' \leq 1 $.
	Using \eqref{eq:ProofAsymp} for $ x\in [0,1) $ we obtain
	\begin{align*}
		\beta_s(x) \beta_s'(x)\leq \dfrac{2}{s-1} (1-x)^{(s+1)/(s-1)} + \dfrac{s-3}{s-1} (1-x)^{-2/(s-1)}.
	\end{align*}
	Since $ \varphi_s' $ is increasing for $ s\geq3 $ we have
	\begin{align*}
		(1-x_s(\varphi))^{-1}\leq \varphi_s'(1)\left( \dfrac{\pi}{2}-\varphi \right)^{-1}.
	\end{align*}
	We then obtain with the previous estimates
	\begin{multline}\label{eq:ProofUniformBound}
		\left( \dfrac{\pi}{2}-\varphi \right)^{(s+1)/(s-1)} \beta_s(x_s(\varphi))\beta_s'(x_s(\varphi))x_s'(\varphi)\\ \leq \dfrac{2}{s-1} \varphi_s'(1)^{(s+1)/(s-1)} + \dfrac{s-3}{s-1}\left( \dfrac{\pi}{2}-\varphi \right)  \varphi'_s(1)^{2/(s-1)}.
	\end{multline}
	One can see that $$ \varphi_s'(1)\leq c(s-1), $$ for some constant $ c>0 $. All in all, the right hand side in \eqref{eq:ProofUniformBound} is uniformly bounded in $ s\geq3 $ and $ \varphi\in [0,\pi/2] $. This implies the uniform bound.
\end{proof}
This completes the proof of the limit of the non-cutoff collision kernel to the hard-sphere kernel. In the next section, we further study the behavior of $b_s$ for $ \theta\to 0 $ when $s\to \infty$. 
\section{Asymptotics of the non-cutoff collision kernel}\label{sec.asymp} We now study the asymptotics of the singular layer of $ b_s(\cos\theta) $ near $ \theta\simeq 0 $ when $ s\to \infty $. To this end, we note that Theorem \ref{thm:Kernel} (ii) in combination with Remark \ref{rem:Thm1} yields
\begin{align*}
		b_s(\cos\theta) \sim \dfrac{1}{s-1}\theta^{-2-2/(s-1)} \sim \dfrac{\theta^{-2}}{s} \qquad \text{as } s\to\infty.
\end{align*}
Thus, we need to look at the scaled function $$ \psi \mapsto b_s(\cos(\psi/\sqrt{s})),$$ with $ \theta= \psi/\sqrt{s} $. In the following, we use this scaling to compute the limit $ s\to \infty $. First, we derive a similar formula to \eqref{eq:DefIntegral2}. Note that  $$\varphi=\frac{\pi}{2}-\frac{\theta}{2}=\frac{\pi}{2}-\frac{\psi}{2\sqrt{s}}.$$ Let us define
\begin{align}\label{eq:DefRescaledFunction}
	\dfrac{\xi_s(\psi)}{2s} := 1-x_{s}\left( \dfrac{\pi}{2}-\dfrac{\psi}{2\sqrt{s}} \right),
\end{align}
where $ \xi_s $ is defined for $ \psi\in[0,\pi\sqrt{s}] $. The inverse function for $ \xi\in[0,2s] $ is given by
\begin{multline}\label{eq:DefRescaledIntegral}
		\psi_s(\xi) = 2\sqrt{s}\left[\dfrac{\pi}{2} -\varphi_s\left(1-\dfrac{\xi}{2s}\right)\right] \\= 2\sqrt{s}\int_{0}^1 \left( \dfrac{1}{\sqrt{1-z^2}}- \dfrac{1}{\sqrt{1-z^{s-1}-\left(1-\frac{\xi}{2s}\right)^2(z^2-z^{s-1})}}\right) \, dz 
		\\ \quad + \dfrac{\xi}{\sqrt{s}} \int_0^1\dfrac{1}{\sqrt{1-z^{s-1}-\left(1-\frac{\xi}{2s}\right)^2(z^2-z^{s-1})}}\, dz.
	\end{multline}
Notice that in the last equality we used the definition of $ \varphi_s $ in \eqref{eq:DefIntegral2}. Note that $ \psi_s $ is an analytic function on $ (0,2s) $. With this we can state the asymptotic behavior.
\begin{theorem}\label{thm:RescaledAsymptoticsKernel}
	The angular part $ b_s(\cos\theta) $, $ s\geq2 $, satisfies the following asymptotic limit
		\begin{align*}
			\lim_{s\to \infty} b_s\left( \cos\left( \dfrac{\psi}{\sqrt{s}} \right)  \right) =  \Phi(\psi),
		\end{align*}
		which holds locally uniformly for $ \psi\in (0,\infty) $. Here, $ \Phi:(0,\infty) \to \R $ is real analytic satisfying 
		\begin{align}\label{eq:ThmAsymptSingularLayerLimit}
			\quad \lim_{\psi\to \infty} \Phi(\psi) = \dfrac{1}{4}.
		\end{align}
		Furthermore, we have
		\begin{align}\label{eq:ThmAsymptSingularLayerExpansion}
				\Phi(\psi)=\dfrac{1}{\psi^2} + \dfrac{1}{\sqrt{\pi}}\, \dfrac{1}{\psi} + \Phi_0(\psi),
		\end{align}
		where $ \Phi_0: [0,\infty)\rightarrow \R $ is continuous.
\end{theorem}
\begin{remark}
	Note that the singularity $ 1/\psi^2 $ of $ \Phi $ for $ \psi\to0 $ is consistent with the asymptotics in Theorem \ref{thm:Kernel} (ii), since $ s C_s \to 1 $ as $ s\to \infty $. Furthermore, the result of the limit $ \psi\to \infty $ coincides with Theorem \ref{thm:Kernel} (i).
\end{remark}
\begin{proof}[Proof of Theorem \ref{thm:RescaledAsymptoticsKernel}]The proof consists of the following 4 steps.\\ \\
	\textit{Step 1.} We first derive the limits
	\begin{align}
			\lim_{s\to \infty} \psi_s(\xi)&=\psi_\infty(\xi) = 2\xi\int_0^\infty \dfrac{1-e^{-\zeta}}{\sqrt{2\zeta}\sqrt{h(\zeta,\xi)}(\sqrt{2\zeta}+\sqrt{h(\zeta,\xi)})} \, d\zeta, \quad \label{eq:ThmIntegralLimit}
			\\
			\lim_{s\to \infty} \psi_s'(\xi)&=\psi_\infty'(\xi) = \int_0^\infty \dfrac{1-e^{-\zeta}}{h(\zeta,\xi)^{3/2}}\, d\zeta,\label{eq:ThmIntegralLimitDerivative}
	\end{align}where $$h(\zeta,\xi)= 2\zeta + \xi\left( 1-e^{-\zeta} \right).$$
	To this end we choose $ \xi \in [0,\infty) $ and assume $ s $ large enough such that $ \xi \in [0,2s] $. Let us write
	\begin{align*}
		1-z^{s-1}-\left(1-\frac{\xi}{2s}\right)^2(z^2-z^{s-1}) = 1-z^2+\left( \dfrac{\xi}{s}-\dfrac{\xi^2}{4s^2}\right) (z^2-z^{s-1})=: g_s(z,\xi).
	\end{align*} 
	Since $ g_s\geq1-z^2 $ the second integral in \eqref{eq:DefRescaledIntegral} goes to zero as $ s\to \infty $. The first term in \eqref{eq:DefRescaledIntegral} can be rearranged to get
	\begin{align*}
			2\sqrt{s} &\int_0^1\dfrac{(  \xi/s - \xi^2/4s^2 )(z^2-z^{s-1}) }{\sqrt{1-z^2}\sqrt{g_s(z,\xi)} (\sqrt{1-z^2} + \sqrt{g_s(z,\xi) }) }\, dz =:I_s(\xi)
	\end{align*}
	We now perform the change of variables $ z=1-\zeta/s $ to get with
	\begin{align*}
		1-\left( 1-\dfrac{\zeta}{s} \right)^2 &= \dfrac{1}{s} \left( 2\zeta- \dfrac{\zeta^2}{s} \right),
		\\
		g_s\left( 1-\dfrac{\zeta}{s}, \xi\right) &= \dfrac{1}{s} \left( 2\zeta- \dfrac{\zeta^2}{s} \right) +\dfrac{1}{s} \left( \xi-\dfrac{\xi^2}{4s} \right) \left( \left( 1-\dfrac{\zeta}{s} \right)^2-\left( 1-\dfrac{\zeta}{s} \right)^{s-1}  \right)\\&=: \dfrac{1}{s} h_s(\zeta,\xi),
	\end{align*}
and	the formula
	\begin{align}\label{eq:ProofAsympIntegral}
		I_s(\xi) = \left( 2\xi - \dfrac{\xi^2}{2s}\right) \int_{0}^s \dfrac{(1-\zeta/s)^2-(1-\zeta/s)^{s-1}}{\sqrt{2\zeta-\zeta^2/s}\sqrt{h_s(\zeta,\xi)}(\sqrt{2\zeta-\zeta^2/s}+\sqrt{h_s(\zeta,\xi)})} \, d\zeta.
	\end{align}
	Using that $ \zeta\leq s $ and $ \xi \leq 2s $ we can obtain
	$$
		2\zeta- \dfrac{\zeta^2}{s}\geq\zeta,$$ and $$  \left( 1-\dfrac{\zeta}{s} \right)^2-\left( 1-\dfrac{\zeta}{s} \right)^{s-1}\ge 0 .$$ Hence, we have
$h_s(\zeta,\xi) \geq \zeta.$
	In addition, we also have
	\begin{align*}
		\left( 1-\dfrac{\zeta}{s} \right)^2-\left( 1-\dfrac{\zeta}{s} \right)^{s-1} \leq \min\left\{1, \dfrac{s-3}{s} \zeta \left( 1-\dfrac{\zeta}{s} \right)^2\right\}\leq  \min\left\{1,\zeta\right\}.
	\end{align*}
	Thus, the integrand in \eqref{eq:ProofAsympIntegral} can be estimated by
	\begin{align*}
		\min \left\lbrace \dfrac{1}{2\sqrt{\zeta}} , \dfrac{1}{2\, \zeta^{3/2}} \right\rbrace .
	\end{align*}
	In conjunction with
	\begin{align*}
			\lim_{s\to \infty} h_s(\zeta,\xi) = 2\zeta + \xi \left( 1- e^{-\zeta}\right) =h(\zeta,\xi)
	\end{align*}
	we conclude the locally uniform convergence 
	\begin{align*}
			\lim_{s\to \infty} \psi_s(\xi) = \psi_\infty(\xi),
	\end{align*}
	where $ \psi_\infty $ is given in \eqref{eq:ThmIntegralLimit}. Since the above estimates also hold in a neighborhood of $ \xi\in(0,\infty) $ in the complex plane, the limit is real analytic. A calculation allows to derive the formula \eqref{eq:ThmIntegralLimitDerivative}. Alternatively, one can compute the derivative of \eqref{eq:DefRescaledIntegral} and mimic the preceding computation. 

	\textit{Step 2.} Since $ \psi_\infty'>0 $ we also have from the analyticity and the locally uniform convergence
	\begin{align*}
		\xi_s(\psi) \to \xi_\infty(\psi) = \psi_\infty^{-1}(\psi), \quad \xi_s'(\psi) \to \xi_\infty'(\psi) = \dfrac{1}{\psi_\infty'(\xi_\infty(\psi))},
	\end{align*}
	locally uniformly for $ \psi\in (0,\infty) $. Furthermore, by \eqref{eq:DefRescaledFunction}
	\begin{align*}
			\lim_{s\to \infty} x_{s}\left( \dfrac{\pi}{2}-\dfrac{\psi}{2\sqrt{s}} \right) = \lim_{s\to \infty} 1- \dfrac{\xi_s(\psi)}{2s} =1.
	\end{align*}
	This yields with the definition of $ b_s(\cos(\psi/s)) $, cf. \eqref{eq:ProofLimitKernel} and formulas \eqref{eq:ProofAsymp},
	\begin{multline*}
		\lim_{s\to \infty} b_s\left( \cos \left( \dfrac{\psi}{\sqrt{s}} \right)  \right) \\= \lim_{s\to \infty} \dfrac{1}{2}\dfrac{1}{\sin(\psi/\sqrt{s})} \dfrac{2}{(s-1)} \dfrac{1}{2} \, \left( 1-x_s\left( \dfrac{\pi}{2} -\dfrac{\psi}{2\sqrt{s}} \right)  \right) ^{-(s+1)/(s-1)} x_s'\left( \dfrac{\pi}{2} -\dfrac{\psi}{2\sqrt{s}} \right) 
		\\
		\qquad+ \lim_{s\to \infty} \dfrac{1}{2}\dfrac{1}{\sin(\psi/\sqrt{s})} \left( 1-x_s\left( \dfrac{\pi}{2} -\dfrac{\psi}{2\sqrt{s}} \right)  \right) ^{-2/(s-1)} x_s'\left( \dfrac{\pi}{2} -\dfrac{\psi}{2\sqrt{s}} \right).
	\end{multline*}
	Using a Taylor expansion we can replace $ \sin(\psi/\sqrt{s}) $ by $ \psi/\sqrt{s} $ without modifying the value of the limit. We use \eqref{eq:DefRescaledFunction} and
	\begin{align*}
			x_s'\left( \dfrac{\pi}{2} -\dfrac{\psi}{2\sqrt{s}} \right) = \dfrac{1}{\sqrt{s}} \xi_s'(\psi),
	\end{align*}
	which is a consequence of \eqref{eq:DefRescaledFunction}, to obtain
	\begin{align}\label{eq:ProofAsympLimit}
		\lim_{s\to \infty} b_s\left( \cos \left( \dfrac{\psi}{\sqrt{s}} \right)  \right) = 	\dfrac{\xi_\infty'(\psi)}{\xi_\infty(\psi)\psi} + \dfrac{\xi_\infty'(\psi)}{2\psi} =: \Phi(\psi).
	\end{align}
	
	\textit{Step 3.} We now use a Taylor approximation for  \eqref{eq:ProofAsympLimit}. It is convenient to define
	\begin{align*}
		\psi_\infty(\xi) = 2 \xi J(\xi), \quad f(\psi) := 2\, \xi_\infty'(\psi)\, J(\xi_\infty(\psi)).
	\end{align*}
	Here, $ J(\xi) $ is the integral in \eqref{eq:ThmIntegralLimit}. This yields
	\begin{align*}
		\xi_\infty(\psi) = \dfrac{\psi}{2 J(\xi_\infty(\psi))}, \quad  \Phi(\psi)  = \dfrac{f(\psi)}{\psi^2} + \dfrac{\xi_\infty'(\psi)}{2\psi}.
	\end{align*}
	We then have 
	\begin{align*}
		\Phi(\psi) = \dfrac{f(0)}{\psi^2} + \dfrac{f'(0)+\xi_\infty'(0)/2}{\psi} + \dfrac{1}{\psi}\left( \dfrac{f(\psi) - f(0)-f'(0)\psi}{\psi} + \dfrac{\xi_\infty'(\psi)-\xi'_\infty(0)}{2}\right),
	\end{align*}
	which defines $ \Phi_0 $. The following formulas hold
	\begin{align}\label{eq:IntegralValuesZero}
		\xi'_\infty(0)=\sqrt{\dfrac{2}{\pi}}, \quad f(0) = 1, \quad f'(0) = \dfrac{\sqrt{2}-1}{\sqrt{2\pi}}.
	\end{align}
	With this we derive $$ f'(0)+\frac{\xi_\infty'(0)}{2} = \dfrac{1}{\sqrt{\pi}},$$ which yields the expression in \eqref{eq:ThmAsymptSingularLayerExpansion}. 
	
	The formulas \eqref{eq:IntegralValuesZero} can be calculated without difficulty, since the integrals are well-defined. For instance,
	\begin{align*}
		2J(0) = \psi_\infty'(0) = \dfrac{1}{\xi_\infty'(0)} = \int_0^\infty \dfrac{1-e^{-\zeta}}{(2\zeta)^{3/2}}\, d\zeta = \int_0^\infty \dfrac{e^{-\zeta}}{\sqrt{2\zeta}}\, d\zeta = \sqrt{\dfrac{\pi}{2}}.
	\end{align*} 
	
	\textit{Step 4.} Finally, for the limit in \eqref{eq:ThmAsymptSingularLayerLimit} we have with \eqref{eq:ProofAsympLimit}
	\begin{align*}
		\lim_{\psi\to \infty} \Phi(\psi)= \lim_{\xi\to \infty} \left(  \dfrac{1}{2\xi^2\, J(\xi) \, \psi_\infty'(\xi)} + \dfrac{1}{4\, \xi\, J(\xi)\, \psi_\infty'(\xi)}\right).
	\end{align*}
	We prove below that
	\begin{align*}
		 \lim_{\xi\to \infty} \sqrt{\xi} \psi_\infty'(\xi)=\lim_{\xi\to \infty} \sqrt{\xi} J(\xi)  =1,
	\end{align*}
	which implies the assertion. For the preceding two limits we use the change of variables $ \zeta=\xi z $ to get
	\begin{align*}
		\sqrt{\xi} \, \psi_\infty'(\xi) = \int_0^\infty \dfrac{1-e^{-\xi z}}{(2z+1-e^{-\xi z})^{3/2}}\, dz.
	\end{align*}
	The integrand can be estimated by (we use here $ \xi\geq1 $ say)
	\begin{align*}
		\min \left\lbrace \dfrac{1}{z^{3/2}} , \dfrac{1}{\sqrt{1-e^{-\xi z}}} \right\rbrace \leq 	\min \left\lbrace \dfrac{1}{z^{3/2}} , \dfrac{1}{\sqrt{1-e^{-z}}} \right\rbrace.
	\end{align*}
	Hence, we can use the dominated convergence theorem to obtain the stated limit. A similar computation applies to $ \sqrt{\xi} J(\xi) $. This concludes the proof.
\end{proof}
This completes the proof of the asymptotics of the singularity for $\theta \simeq 0$ as $s\to \infty$. In the next section, we provide a proof of the limit of solutions to the spatially homogeneous Boltzmann equation without cutoff to solutions of the homogeneous Boltzmann equation for hard-spheres using the estimates in Sections~\ref{sec.kernel} and \ref{sec.asymp}.

\section{Convergence of the solution for the homogeneous Boltzmann equation}\label{sec:HomogBE}
In this section, we consider the spatially homogeneous Boltzmann equation
\begin{align}\label{eq:HomogBE}
	\partial_t f = Q(f,f), \quad f(0,\cdot) = f_0(\cdot)
\end{align}
with collision kernel $ B_s(|v-v_*|, n\cdot \sigma) $, $ s>2 $, given in \eqref{eq:CollisionKernel}. Let us first recall the following well-posedness result for cutoff kernels with hard potentials $ \gamma\in(0,1] $ (e.g. hard-sphere corresponding to $ s=\infty $), see \cite[Theorem 1.1]{MischlerWennberg1999HomogBE} and \cite[Section 3.7, Theorem 3]{Villani2002Review}. The first well-posedness results are due to Arkeryd \cite{Arkeryd1972BoltzmannEqI,Arkeryd1972BoltzmannEqII}. We use here the weighted spaces $ L^1_p $ with weight function $ (1+|v|^2)^{p/2} $.
\begin{lemma}\label{lem:CutoffWellPosed}
	Let $ f_0\in L^1_2 $, then there is a unique solution $ f\in C([0,\infty); L^1_2) $ to \eqref{eq:HomogBE} which preserves energy, i.e. for all $ t\geq0 $
	\begin{align*}
		\int_{\R^3} |v|^2\, f(t,v)\, dv = \int_{\R^3} |v|^2\, f_0(v)\, dv.
	\end{align*}
\end{lemma}
\begin{remark}
	Let us mention that the condition of the energy conservation is essential for uniqueness \cite{LuWennberg2002SolutionIncEnergy,Wennberg1999ExampleNonuniqueness}.
\end{remark}

Next, we consider the non-cutoff kernel $ B_s $. Since we are interested in the limit $ s\to\infty $, we can assume $ s>5 $ so that
\begin{align}\label{eq:NonCutoffKernelAssumpt}
		\gamma(s)= \dfrac{s-5}{s-1}>0,\quad \int_0^\pi \theta \, b_s(\cos\theta) \sin \theta\, d\theta \leq c_0,
\end{align}
where the constant $ c_0 $ is independent of $ s>5 $, see Theorem \ref{thm:Kernel} (iii). In this case, we can use the weak formulation of \eqref{eq:HomogBE} by testing with functions $ \psi\in C^1_b( [0,\infty)\times \R^3) $, see e.g. \cite[Section 4.1]{Villani2002Review}. The collision operator can be define by means of the pre-postcollisional change of variables
\begin{align*}
		\int_{\R^3} Q_s(f,f)(v)\, \psi(v)\, dv = \int_{\R^3} \int_{\R^3} |v-v_*|^{\gamma} f f_* \int_{S^2} b_s(\cos\theta)\, (\psi'-\psi)\, d\sigma dv_*dv.
\end{align*}
For the integral on the sphere we have, via a Taylor approximation,
\begin{align*}
	\left| \int_{S^2} b_s(\cos\theta)\, (\psi'-\psi)\, d\sigma \right| \leq C_0\norm[C^1(\R^3)]{\psi} |v-v_*|,
\end{align*}
for some constant $ C_0>0 $ independent of $ s>5 $. Let us also define the entropy of $ f $
\begin{align*}
		H(f) = \int_{\R^3} f\ln f \, dv.
\end{align*}
We also recall the existence of weak solutions to the homogeneous Boltzmann equation, which is the content of the following lemma, see e.g. \cite[Section 4]{Villani1998NewClassWeakSol} and \cite[Section 4.7, Theorem 9 (ii)]{Villani2002Review}. With a slight abuse of notation we write $f^s_t(v):=f^s(t,v)$ and $f^\infty_t(v):=f^\infty(t,v)$ to describe the solutions to the Boltzmann equations with kernels $B_s$ and $B_\infty$, respectively.
\begin{lemma}\label{lem:NonCutoffEx}
	Let $ f_0\in L^1_{1+\gamma+\delta} $, for $ \delta>0 $ arbitrary, with finite entropy. Under the conditions \eqref{eq:NonCutoffKernelAssumpt} there is a weak solution $ f^s\in L^\infty([0,\infty); L^1_{1+\gamma+\delta}) $ to \eqref{eq:HomogBE} which preserves energy. Furthermore, we have $ H(f^s_t)\leq H(f_0) $ for all $ t\geq0 $.
\end{lemma}
We finally have the following convergence result. 
\begin{theorem}\label{thm.solution}
	Let $ f_0\in L^1_p $ with finite entropy and arbitrary $ p>2 $. Consider a sequence of weak solutions $ f^s $ to \eqref{eq:HomogBE} as in Lemma \ref{lem:NonCutoffEx} with collision kernel $ B_s $, $ s>5 $. Then, $ f^s_t\rightharpoonup f^\infty_t $ weakly in $ L^1 $ for all $ t\geq0 $ as $ s\to \infty $, where $ f^\infty $ is the unique solution to \eqref{eq:HomogBE} for hard-sphere interactions.
\end{theorem}
\begin{proof}[Proof of Theorem \ref{thm.solution}]
	First of all, applying a version of the Povzner estimate (see e.g. \cite[Lemma 2.2]{MischlerWennberg1999HomogBE} which is also applicable for non-cutoff kernels, cf. \cite[Appendix]{Villani2002Review}) we have
	\begin{align}\label{eq:UniformMomentBd}
		\sup_{t\in [0,\infty)} \norm[L^1_p]{f^s_t} \leq C(\norm[L^1_p]{f_0})=:C_p.
	\end{align}
	This estimate is independent of $s$ as long as $s $ is sufficiently large. Assume for example $ s>6 $. In fact, in the Povzner estimate we only need a uniform lower and upper bound on the angular part $ b_s(\cos\theta) $. This is ensured by Theorem \ref{thm:Kernel} items (i) and (iii). Also note that for, say, $ s>6 $ we have $ \gamma(s)\geq1/5 $. Furthermore, from the weak formulation we also obtain
	\begin{align*}
			\left| \int_{\R^3} \psi(v) f^s_{t_1}(v)\, dv - \int_{\R^3} \psi(v) f^s_{t_2}(v)\, dv \right| \leq C\norm[C^1]{\psi} |t_1-t_2|,
	\end{align*}
	for all $ t_1,\, t_2\geq0 $. Here, the constant $ C $ is independent of $ s>6 $ due to \eqref{eq:NonCutoffKernelAssumpt} and \eqref{eq:UniformMomentBd}. By the uniform entropy bound $$ H(f^s_t)\leq H(f_0), $$ and the previous weak equicontinuity property we can apply the Dunford-Pettis theorem yielding $$ f^{s_n}_t\rightharpoonup f^\infty_t,$$ weakly in $ L^1 $ for all $ t\geq0 $ for a subsequence $ s_n\to \infty $. 
 
	 Using Theorem \ref{thm:Kernel}, items (i) and (iii), we can pass to the limit in the weak formulation. Hence, $ f^\infty $ is a weak solution to \eqref{eq:HomogBE} for hard-sphere interactions. Since there is no angular singularity, one can infer $$ f^\infty\in C([0,\infty),L^1_2) .$$ By the uniform moment bound \eqref{eq:UniformMomentBd}, the second moments also converge for all $ t\geq0 $ as $ s_n\to \infty $. As a consequence $ f^\infty $ preserves energy and thus $ f^\infty $ is the unique solution in Lemma \ref{lem:CutoffWellPosed}. This implies that the whole sequence converges $ f^s_t\rightharpoonup f^\infty_t  $ as $ s\to \infty $.
\end{proof}

\subsection{Conclusion} We proved the convergence of the collision kernel for inverse power law interactions $ 1/r^{s-1} $ to the hard-sphere kernel as $ s\to \infty $. We furthermore studied the asymptotics of the angular singularity $ \theta\to 0 $. Finally, solutions to the homogeneous Boltzmann equation converge respectively.

\section*{Acknowledgement} The authors gratefully acknowledge the support by the Deutsche Forschungsgemeinschaft (DFG, German Research Foundation) through the collaborative research centre The mathematics of emerging effects (CRC 1060, Project-ID 211504053). J. W. Jang is supported by the National Research Foundation of Korea (NRF) grant funded by the Korean government (MSIT) NRF-2022R1G1A1009044 and by the Basic Science Research Institute Fund of Korea NRF-2021R1A6A1A10042944. B. Kepka is funded by the Bonn International Graduate School of Mathematics at the Hausdorff Center for Mathematics (EXC 2047/1, Project-ID 390685813). J. J. L. Vel\'azquez is also funded by DFG under Germany's Excellence Strategy-EXC-2047/1-390685813.

	\bibliographystyle{hplain}
	\bibliography{main}
\end{document}